\documentclass[12pt,a4]{amsart}
\usepackage{amsmath,amsthm,amsfonts,graphicx,float,verbatim,tikz}
\usepackage{cleveref}
\usepackage[margin=1in]{geometry}

\newtheorem{theorem}{Theorem}
\newtheorem{lem}[theorem]{Lemma}
\newtheorem{proposition}[theorem]{Proposition}

\newcommand{\G}{\mathcal{G}}
\newcommand{\ol}[1]{\overline{#1}}

\tikzstyle{vertex}=[circle, draw, fill=black, inner sep=0pt, minimum width=6pt]
\tikzstyle{pedge}=[draw,-]

\DeclareMathOperator{\Aut}{Aut}

\title{Symmetric Perfect $2$-colorings on $J(10,3)$}
\author[Paul Tricot]{Paul Tricot}

\begin{document}

\maketitle
\begin{center}
   \address{School of Information Science, Tohoku University, Japan} 
\end{center}

\section{Introduction}
A perfect $2$-coloring of a Johnson graph can be associated to one of the non-principal eigenvalue of the graph $\theta_1 > \dots > \theta_k$. The perfect $2$-colorings of the Johnson graphs $J(n,3)$ associated with $\theta_1$ have been characterized by Meyerowitz \cite{7}, and those associated with $\theta_3$ by Martin \cite{6}. Those associated with $\theta_2$ have been studied by several mathematicians. Evans, Gavrilyuk, Goryainov and Vorob'ev \cite{1,3} classified them for $n$ odd and for $n>10$. Avgustinovich and Mogilnykh also studied these perfect $2$-colorings \cite{2,5}, in particular for $n=6$, 7 and 8.\\

In \cite{1}, Gavrilyuk and Goryainov proved that a perfect $2$-coloring of $J(n,3)$ associated with $\theta_2$ and symmetric quotient matrix is possible only when $n \in \{6, 10\}$. In this paper, we survey the known constructions in the case $n=6$, we give a new construction for the two known perfect $2$-colorings in the case $n=10$, and prove that these are the only possible ones.

\section{Preliminaries}

A perfect $m$-coloring of a regular graph $\Gamma$ is a partition $P_1, \dots, P_m$ of the vertices such that there exist fixed numbers $p_{i,j}$ ($i,j \in [m]$) that verify
$$ \forall x \in P_i, |\Gamma(x) \cap P_j| = p_{i,j},$$
where $\Gamma(x)$ is the neighborhood of $x$ in $\Gamma$, which is the set of all its neighbors. The matrix $P = [p_{i,j}]_{1 \leq i,j \leq m}$ is called the quotient matrix of the coloring. This means that for every vertex $x \in P_i$, $x$ has exactly $p_{i,j}$ neighbors in $P_j$. We will say that the perfect coloring is symmetric when the quotient matrix is.\\

$P$ has $m$ eigenvalues that are among the eigenvalues of $\Gamma$, by that we mean eigenvalues of its adjacency matrix. In particular the valency of the graph is always an eigenvalue of $P$ (see \cite{1}).\\

The Johnson graph $J(n,k)$ with parameters $n, k \in \mathbb{N} $ has as vertices the subsets of $[n]:= \{ 1,\dots,n \}$ of size $k$. The vertices $x$ and $y$ are connected when $|x \cap y| = k-1$. The distance between two elements is given by $d(x,y) = k - |x \cap y|$. The graph $J(n,k)$ is regular, with valency $k(n-k)$.\\

The eigenvalues of the Johnson graph $J(n,k)$ are the eigenvalues of its adjacency matrix, and are well known to be $\theta_i := (k-i)(n-k-i)-i$, $i \in \{0,\dots,k\}$. The quotient matrix $P$ of a perfect $2$-coloring of $J(n,k)$ always has as eigenvalue $\theta_0$ which is the valency of the graph, and $\theta_i$ for some $i \in \{1,\dots,k\}$. In this paper we focus on perfect $2$-colorings with eigenvalues $\theta_0$ and $\theta_2$.\\

Under certain conditions on the quotient matrix $P$, it is possible to join the parts of an $m$-coloring to form a coloring with less parts (see \cite[Lemma 1]{2}). For instance, consider a perfect coloring with three parts $P_1, P_2, P_3$ and quotient matrix $P$ that we want to merge into a perfect $2$-coloring with parts $P_1 \cup P_2$ and $P_3$. Then we need to make sure that for $x \in P_1$ and $y \in P_2$, $|\Gamma(x) \cap ( P_1 \cup P_2 ) | = |\Gamma(y) \cap ( P_1 \cup P_2 ) |$, and also $|\Gamma(x) \cap P_3 | = |\Gamma(y) \cap P_3 |$. This is equivalent to $p_{1,1} + p_{1,2} = p_{2,1} + p_{2,2}$ and $p_{1,3} = p_{2,3}$. But since the row sum of $P$ is fixed, only one of these equality is necessary.\\

In the general case, consider a perfect $m$-coloring with parts $P_1, \dots, P_m$, and $C_1, \dots, C_l$ a partition of $[m]$. Then $$\bigcup_{i \in C_1} P_i, \dots, \bigcup_{i \in C_l} P_i$$ is a perfect $l$-coloring if and only if each of the submatrices $ [p_{x,y}]_{x \in C_i, y \in C_j} $ for $i \in [l]$, $ j \in [l-1]$ have constant row sum.\\

One way of finding perfect colorings on the Johnson graph $J(n,k)$ is by the orbit construction method (see \cite[Section 3]{2}). Consider a graph $G$ on $n$ vertices. If there are $m$ orbits of $\Aut(G)$ acting on the set of vertices of $J(n,k)$, then the orbits form a $m$-coloring. Depending on the graph $G$ on $n$ vertices that was chosen for the construction, the $m$-coloring obtained can be merged into a 2 or $3$-coloring. Most of the perfect colorings in the literature are constructed using this method \cite{2,3}.\\

\section{Symmetric Perfect $2$-colorings on $J(6,3)$}

In \cite{1}, Gavrilyuk and Goryainov proved that a symmetric perfect $2$-coloring with eigenvalues $\theta_0$, $\theta_2$ on $J(n,3)$ is possible only when $n \in \{6, 10\}$, and in this case the quotient matrix can only be $ \begin{bmatrix} 2n-8 & n-1\\ n-1 & 2n-8 \end{bmatrix} $. In \cite{5}, Avgustinovich and Mogilnykh showed the following construction for the case $n=6$.\\

The graph $J(6,3)$ is antipodal of diameter 3, which means that for any vertex $v$ there is a unique vertex at distance 3 from $v$. Two such vertices are called antipodal vertices. $J(6,3)$ can be partitioned into 10 pairs of antipodal vertices, which forms a perfect $10$-coloring with quotient matrix $J-I$ (where $J$ is the all 1 matrix and $I$ the identity matrix). This perfect $10$-coloring can be merged into a perfect $2$-coloring with quotient matrix $ \begin{bmatrix} 4 & 5\\ 5 & 4 \end{bmatrix} $ by taking any two groups of five pairs each.\\

$J(6,3)$ is small enough that a computer search can be used to list all possible symmetric perfect $2$-colorings. It turns out that the only possible ones are those mentioned above, and that they are all isomorphic to one of the two perfect $2$-colorings $\{X_1,X_2\}$ and $\{X'_1,X'_2\}$ which induced subgraphs are represented below.

\vspace{1cm}
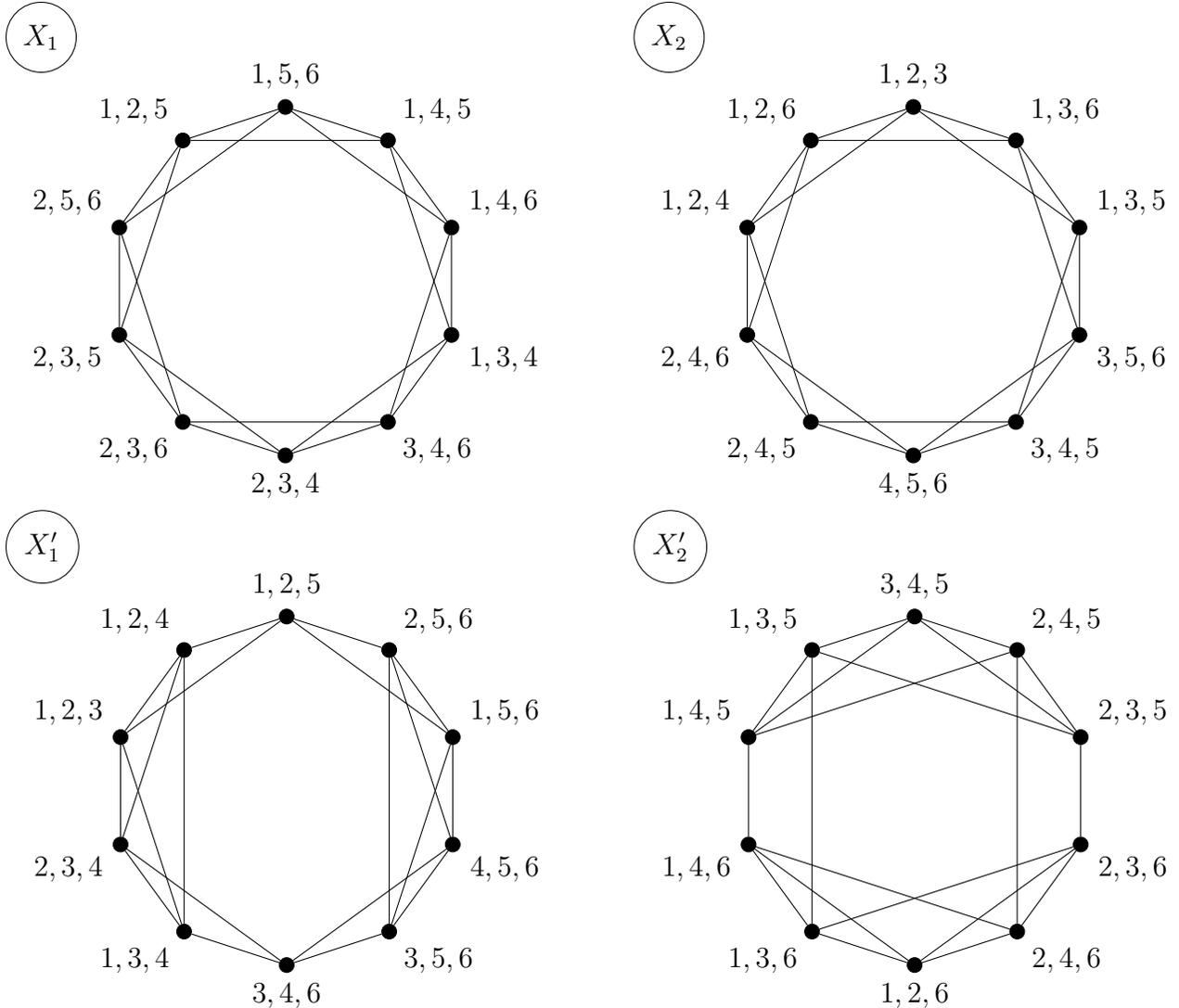
\begin{figure}[H]
\centering
\begin{tikzpicture}



        \draw (-1,6) node[draw, circle] {$X_1$};
	\foreach \pos/\name/\lab/\a in 
    	{{(2.5,5)/P1/$1,5,6$/90},
            {(3.97,4.52)/P2/$1,4,5$/54},
    	{(4.88,3.27)/P3/$1,4,6$/18},
    	{(4.88,1.73)/P4/$1,3,4$/342},
    	{(3.97,0.48)/P5/$3,4,6$/306},
    	{(2.5,0)/P6/$2,3,4$/270},
    	{(1.03,0.48)/P7/$2,3,6$/234},
    	{(0.12,1.73)/P8/$2,3,5$/198},
    	{(0.12,3.27)/P9/$2,5,6$/162},
    	{(1.03,4.52)/P10/$1,2,5$/126}}
    		\node[vertex,label=\a:\lab] (\name) at \pos {};
    		
    	\foreach \source/\dest in
        	{P1/P2, P2/P3, P3/P4, P4/P5, P5/P6, P6/P7, P7/P8, P8/P9, P9/P10, P10/P1, P1/P3, P2/P4, P3/P5, P4/P6, P5/P7, P6/P8, P7/P9, P8/P10, P9/P1, P2/P10}
        		\path[pedge] (\source) -- (\dest);

        \draw (8,6) node[draw, circle] {$X_2$};
        \foreach \pos/\name/\lab/\a in 
    	{{(11.5,5)/P11/$1,2,3$/90},
            {(12.97,4.52)/P12/$1,3,6$/54},
    	{(13.88,3.27)/P13/$1,3,5$/18},
    	{(13.88,1.73)/P14/$3,5,6$/342},
    	{(12.97,0.48)/P15/$3,4,5$/306},
    	{(11.5,0)/P16/$4,5,6$/270},
    	{(10.03,0.48)/P17/$2,4,5$/234},
    	{(9.12,1.73)/P18/$2,4,6$/198},
    	{(9.12,3.27)/P19/$1,2,4$/162},
    	{(10.03,4.52)/P20/$1,2,6$/126}}
    		\node[vertex,label=\a:\lab] (\name) at \pos {};
    		
    	\foreach \source/\dest in
        	{P11/P12, P12/P13, P13/P14, P14/P15, P15/P16, P16/P17, P17/P18, P18/P19, P19/P20, P20/P11, P11/P13, P12/P14, P13/P15, P14/P16, P15/P17, P16/P18, P17/P19, P18/P20, P19/P11, P12/P20}
        		\path[pedge] (\source) -- (\dest);

\end{tikzpicture}

\begin{tikzpicture}
        \draw (-1,6) node[draw, circle] {$X'_1$};
	\foreach \pos/\name/\lab/\a in 
    	{{(2.5,5)/P21/$1,2,5$/90},
            {(3.97,4.52)/P22/$2,5,6$/54},
    	{(4.88,3.27)/P23/$1,5,6$/18},
    	{(4.88,1.73)/P24/$4,5,6$/342},
    	{(3.97,0.48)/P25/$3,5,6$/306},
    	{(2.5,0)/P26/$3,4,6$/270},
    	{(1.03,0.48)/P27/$1,3,4$/234},
    	{(0.12,1.73)/P28/$2,3,4$/198},
    	{(0.12,3.27)/P29/$1,2,3$/162},
    	{(1.03,4.52)/P30/$1,2,4$/126}}
    		\node[vertex,label=\a:\lab] (\name) at \pos {};
    		
    	\foreach \source/\dest in
        	{P21/P22, P22/P23, P23/P24, P24/P25, P25/P26, P26/P27, P27/P28, P28/P29, P29/P30, P30/P21, P21/P23, P22/P24, P22/P25, P23/P25, P24/P26, P26/P28, P27/P29, P27/P30, P28/P30, P29/P21}
        		\path[pedge] (\source) -- (\dest);

        \draw (8,6) node[draw, circle] {$X'_2$};
        \foreach \pos/\name/\lab/\a in 
    	{{(11.5,5)/P31/$3,4,5$/90},
            {(12.97,4.52)/P32/$2,4,5$/54},
    	{(13.88,3.27)/P33/$2,3,5$/18},
    	{(13.88,1.73)/P34/$2,3,6$/342},
    	{(12.97,0.48)/P35/$2,4,6$/306},
    	{(11.5,0)/P36/$1,2,6$/270},
    	{(10.03,0.48)/P37/$1,3,6$/234},
    	{(9.12,1.73)/P38/$1,4,6$/198},
    	{(9.12,3.27)/P39/$1,4,5$/162},
    	{(10.03,4.52)/P40/$1,3,5$/126}}
    		\node[vertex,label=\a:\lab] (\name) at \pos {};
    		
    	\foreach \source/\dest in
        	{P31/P32, P32/P33, P33/P34, P34/P35, P35/P36, P36/P37, P37/P38, P38/P39, P39/P40, P40/P31, P31/P33, P31/P39, P32/P35, P32/P39, P33/P40, P34/P36, P34/P37, P35/P38, P36/P38, P37/P40}
        		\path[pedge] (\source) -- (\dest);

\end{tikzpicture}
\caption{Induced subgraphs of symmetric perfect $2$-colorings of $J(6,3)$}
\end{figure}

It is interesting to consider the stabilizers of those perfect colorings in the automorphism group of $J(6,3)$, which is known to be induced by the symmetric group on $\{1, \dots, 6\}$ and the complementation map (which maps $3$-subsets to their complement in $\{1, \dots, 6\}$) (see \cite{8}). Since the perfect $2$-colorings above are constructed from pairs of antipodal vertices, they are invariant by the complementation map. Let $G$ be the group of automorphisms of $J(6,3)$ induced by the symmetric group, $S$ be the stabilizer of $\{X_1,X_2\}$ in $G$, and $S'$ the stabilizer of $\{X'_1,X'_2\}$.\\

The subgraph induced by $X_1$ contains 10 triangles that can be separated into two types:
\begin{itemize}
    \item $\{\{a,b,c\}, \{a,b,d\}, \{a,b,e\} \}$, where $a,b,c,d,e$ are distinct elements of $\{1, \dots, 6\}$,
    \item $\{\{a,b,c\}, \{a,b,d\}, \{a,c,d\} \}$, where $a,b,c,d$ are distinct elements of $\{1, \dots, 6\}$.
\end{itemize}
Every element $f \in S$ must induce an automorphism of $X_1$, and thus map each triangle to a triangle of the same type. There are 5 triangles of the first type, which are connected by a common vertex in the shape of a 5-cycle. It follows that $S$ is isomorphic to the automorphism group of a 5-cycle, which is the dihedral group with 10 elements.\\

The subgraph induced by $X'_1$ contains two cliques (complete subgraph) of size 4, of two different types:
\begin{itemize}
    \item $C_1 := \{\{a,b,c\} \in \{1,2,3,4\} \mid a,b,c \text{ distinct}\}$,
    \item $C_2 := \{\{x,5,6\} \mid x \in \{1,2,3,4\}\}$.
\end{itemize}
Every element $f \in S'$ must induce an automorphism of $X'_1$, and thus map each clique to itself. Thus
\begin{align*}
    S' = &\Bigl\{f \in G \mid f(C_1)=C_1, f(C_2)=C_2, f(\{\{1,2,5\}, \{3,4,6\}\})=\{\{1,2,5\}, \{3,4,6\}\}\Bigr\}\\
    = &\Bigl\{f \in G \mid f(\{1,2,3,4\})=\{1,2,3,4\}, f(\{5,6\})=\{5,6\},\\
    & f(\{\{1,2,5\}, \{3,4,6\}\})=\{\{1,2,5\}, \{3,4,6\}\}\Bigr\}\\
    = &\Bigl\{f \in G \mid f(\{1,2,3,4\})=\{1,2,3,4\}, f(\{5,6\})=\{5,6\},\\
    & f(\{1,2,5\})=\{1,2,5\}, f(\{3,4,6\})=\{3,4,6\}\Bigr\}\\
    &\cup \Bigl\{f \in G \mid f(\{1,2,3,4\})=\{1,2,3,4\}, f(\{5,6\})=\{5,6\}, f(\{1,2,5\})=\{3,4,6\},\\
    & f(\{3,4,6\})=\{1,2,5\}\Bigr\}\\
    = &\Bigl\{f \in G \mid f(\{1,2\})=\{1,2\}, f(\{3,4\})=\{3,4\}, f(5)=5, f(6)=6\Bigr\}\\
    &\cup \Bigl\{f \in G \mid f(\{1,2\})=\{3,4\}, f(\{3,4\})=\{1,2\}, f(5)=6, f(6)=5\Bigr\}\\
    = &\{ Id, (1,2), (3,4), (1,2)(3,4), (1,3)(2,4)(5,6), (1,4)(2,3)(5,6), (1,3,2,4)(5,6),\\
    & (1,4,2,3)(5,6)\}\\
    = &\langle (1,2) , (1,3,2,4)(5,6) \rangle.
\end{align*}
This is isomorphic to the dihedral group with 8 elements.\\

\section{Symmetric Perfect $2$-colorings on $J(10,3)$}

 It is mentioned in \cite{1} that for $J(10,3)$ there are only two non-isomorphic perfect $2$-colorings with the symmetric quotient matrix $ \begin{bmatrix} 12 & 9\\ 9 & 12 \end{bmatrix} $. But since a formal proof has never been written, we will attempt to do it here by extending the method used in \cite{3} by R.J. Evans, A.L. Gavrilyuk, S. Goryainov and K. Vorob’ev.\\

There are two known non-isomorphic constructions of perfect $2$-colorings of $J(10,3)$ with the above mentioned symmetric quotient matrix. One of the construction was found by Gavrilyuk and Goryainov (but to our knowledge does not appear in any publication), by using the orbit construction method from two 5-cycles. The second construction can be found in \cite[Construction 3]{2}, using the same method from a complete bipartite graph with parts of size 5, from which we remove a perfect matching.\\

What follows is a different construction for those two $2$-colorings.\\

Let $\G$ be the cycle graph on 10 vertices, and consider the action of $\Aut(\G)$ on $J(10,3)$. The group $\Aut(\G)$ is known as the dihedral group of order 20 consisting of 10 rotations (powers of the cycle permutation $(1, 2,\dots,10)$ ) and 10 reflections. For instance it is generated by the two permutations $(1, 2, \dots, 10)$ and $(2, 10)(3, 9)(4, 8)(5, 7)$. Thus the eight orbits of $\Aut(G)$ acting on $J(10,3)$ are:
\begin{itemize}
    \item $A := \{ \{ a,b,c \} \in J(10,3) \mid d(a,b) = 1, d(b,c) = 1, d(a,c) = 2\}$,
    \item $B := \{ \{ a,b,c \} \in J(10,3) \mid d(a,b) = 1, d(b,c) = 2, d(a,c) = 3\}$,
    \item $C := \{ \{ a,b,c \} \in J(10,3) \mid d(a,b) = 1, d(b,c) = 3, d(a,c) = 4\}$,
    \item $D := \{ \{ a,b,c \} \in J(10,3) \mid d(a,b) = 1, d(b,c) = 4, d(a,c) = 5\}$,
    \item $E := \{ \{ a,b,c \} \in J(10,3) \mid d(a,b) = 2, d(b,c) = 2, d(a,c) = 4\}$,
    \item $F := \{ \{ a,b,c \} \in J(10,3) \mid d(a,b) = 2, d(b,c) = 3, d(a,c) = 5\}$,
    \item $G := \{ \{ a,b,c \} \in J(10,3) \mid d(a,b) = 2, d(b,c) = 4, d(a,c) = 4\}$,
    \item $H := \{ \{ a,b,c \} \in J(10,3) \mid d(a,b) = 3, d(b,c) = 3, d(a,c) = 4\}$.\\
\end{itemize}

These orbits corresponds to the 3 ``types'' of triple of points in the 10-cycle:\\

\begin{figure}[H]
    \centering
    \includegraphics{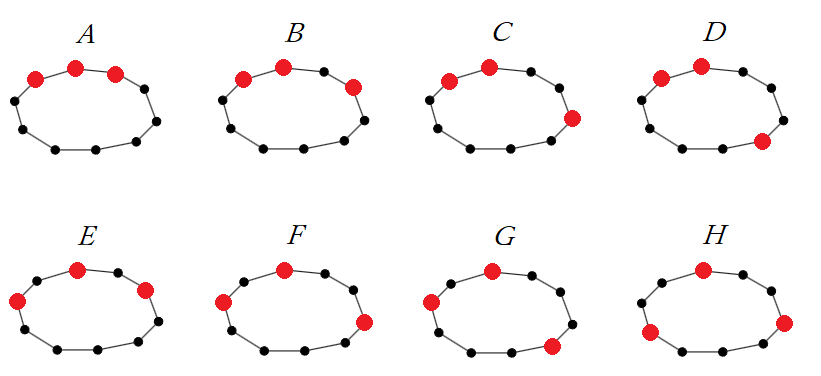}
    \caption{Visualisation of the orbits of $\Aut(\G)$ acting on $J(10,3)$}
    \label{fig:3-orb-in-10-cycle}
\end{figure}

This is a perfect $8$-coloring of $J(10,3)$ with quotient matrix
$$
\begin{bmatrix}
    2&6&4&4&2&2&1&0\\
    3&3&4&2&2&4&1&2\\
    2&4&3&4&1&2&2&3\\
    2&2&4&5&1&4&2&1\\
    2&4&2&2&2&4&4&1\\
    1&4&2&4&2&5&1&2\\
    1&2&4&4&4&2&2&2\\
    0&4&6&2&1&4&2&2\\
\end{bmatrix}
$$
that can be merged in two ways into perfect $2$-colorings. First by $ P_1 := A \cup B \cup C \cup H $ and $ P_2 := D \cup E \cup F \cup G $, and secondly by $ P_1' := C \cup D \cup G \cup H $ and $ P_2' := A \cup B \cup E \cup F $. Those two perfect $2$-colorings are not isomorphic since it can be computed that the subgraph of $J(n,3)$ induced by $P_1$ have different eigenvalues than the one induced by $P_1'$ or $P_2'$.\\

\section{Classification}

We now want to prove that $\{P_1,P_2\}$ and $\{P'_1,P'_2\}$ are the only perfect $2$-colorings (up to isomorphism). We will use the notations and tools of \textup{\cite{1}} and \textup{\cite{3}}. In the rest of this section, we consider a perfect coloring of $J(10,3)$ into two parts $X_1, X_2$ such that the quotient matrix is $ \begin{bmatrix} 12 & 9\\ 9 & 12 \end{bmatrix} $. A simple counting argument shows that $|X_1|=|X_2|=60$.\\

We will denote $abc := \{a,b,c\}$, $ab* := \{a,b,x \mid x \in [10] \setminus \{a,b\} \}$, and denote the intersection with $X_1$ in this way: $\ol{abc} = 1$ if $abc \in X_1$ and $0$ otherwise, and for $S \subseteq J(10,3)$, $\ol{S} := | S \cap X_1 |$. \\

The method consists in looking at how the neighborhood of a point can be distributed between the two parts. The subgraph of $J(10,3)$ induced by the neighborhood of a point $abc$ is isomorphic to a 3 by 7 grid, and we can represent its distribution among $X_1$ and $X_2$ by the nb-array \cite{3}: 
$$
\begin{matrix}
                    & \ol{abd} & \ol{abe} & \ol{abf} & \dots & \dots & \dots & \dots & \gets ab\text{-row}\\
    N(abc) : & \ol{acd} & \ol{ace} & \ol{acf} & \dots & \dots & \dots & \dots & \gets ac\text{-row}\\
                    & \ol{bcd} & \ol{bce} & \ol{bcf} & \dots & \dots & \dots & \dots & \gets bc\text{-row}\\
                    & \uparrow & \uparrow & \uparrow &\\
                    & d & e & f & \cdots
\end{matrix}
$$
The order of the rows and columns is arbitrary.\\

Most of the proof will rely on the following lemma from \textup{\cite{3}}.

\begin{lem}[\cite{3}] \label{lma1}
    For any five distinct elements $a,b,c,d,e \in [10]$, we have
    $$
        \ol{ab*} - \ol{ac*} = 3 ( \ol{abd} + \ol{abe} + \ol{cde} - \ol{acd} - \ol{ace} - \ol{bde} ).
    $$
\end{lem}

In particular, since $\ol{ab*} - \ol{abc}$ is the row sum of the row $ab$ in $N(abc)$, we can see that the difference between two row sums in $N(abc)$ is always a multiple of 3. Moreover each row consists of seven $0$ and $1$, so each row sum is between $0$ and $7$, and the difference between two row sums is $0$, $3$ or $6$. And since when $\ol{abc} = 1$ the total sum of the nb-array must be 12, we have the following.

\begin{lem} \label{lma2}
    For $abc \in X_1$, the multiset of row sums of $N(abc)$ is among $$\{3,3,6\}, \{4,4,4\}, \{2,5,5\}, \{0,6,6\}, \{1,4,7\} .$$
\end{lem}

We also have as a direct consequence of \Cref{lma1} :

\begin{lem}
    For any $a,b,c,d \in [10]$ with $a \ne b$ and $c \ne d$, $\ol{ab*} \equiv \ol{cd*} \pmod 3$.
\end{lem}

Because of this, we can define the type of a part $X_1$ of a partition to be $k \in \{0,1,2\}$ if for any distinct $a,b \in [10]$, $\ol{ab*} \equiv k \pmod 3$. When $\ol{abc} = 1$, a row sum in $N(abc)$ must be $k-1 \pmod 3$. So we can further restrict the possibilities in \Cref{lma2}.

\begin{lem}
\begin{enumerate}
    \item If $X_1$ is of type 0, for $abc \in X_1$ the row sums of $N(abc)$ is $\{2,5,5\}.$
    \item If $X_1$ is of type 1, for $abc \in X_1$ the row sums of $N(abc)$ is $\{3,3,6\}$ or $\{0,6,6\}.$
    \item If $X_1$ is of type 2, for $abc \in X_1$ the row sums of $N(abc)$ is $\{4,4,4\}$ or $\{1,4,7\}.$
\end{enumerate}
\end{lem}

We can also use \Cref{lma1} to eliminate some 2 by 2 patterns that can not appear in $N(abc)$. For a row $r$ of $N(abc)$, denote by $\ol{r}$ the sum of the row $r$. If we fix the order of rows and columns then nb-arrays can be seen as matrices, so we will use matrices in the next lemma for convenience.\\

The next lemma is an extension of \textup{\cite[Lemma 5.4]{3}}.

\begin{lem} \label{lma4}
    Let $abc \in J(10,3)$ and $r_1, r_2$ be two rows of $N(abc)$ such that $\ol{r_1} \ge \ol{r_2}$. Then
    \begin{enumerate}
        \item $\begin{bmatrix} 0 & 0\\ 1 & 1 \end{bmatrix}$ is not a submatrix of $\begin{bmatrix} r_1\\ r_2 \end{bmatrix}$,
        \item If $\ol{r_1} - \ol{r_2} \ge 3$ then $\begin{bmatrix} 0 & 0\\ 0 & 1 \end{bmatrix}$, $\begin{bmatrix} 0 & 0\\ 1 & 0 \end{bmatrix}$, $\begin{bmatrix} 1 & 0\\ 1 & 1 \end{bmatrix}$, $\begin{bmatrix} 0 & 1\\ 1 & 1 \end{bmatrix}$ are not submatrices of $\begin{bmatrix} r_1\\ r_2 \end{bmatrix}$,
        \item If $\ol{r_1} - \ol{r_2} = 6$ then $\begin{bmatrix} 0 & 0\\ 0 & 0 \end{bmatrix}$, $\begin{bmatrix} 1 & 1\\ 1 & 1 \end{bmatrix}$, $\begin{bmatrix} 1 & 0\\ 0 & 1 \end{bmatrix}$, $\begin{bmatrix} 0 & 1\\ 1 & 0 \end{bmatrix}$, $\begin{bmatrix} 1 & 0\\ 1 & 0 \end{bmatrix}$, $\begin{bmatrix} 0 & 1\\ 0 & 1 \end{bmatrix}$ are not submatrices of $\begin{bmatrix} r_1\\ r_2 \end{bmatrix}$.\\
    \end{enumerate}
\end{lem}

From this lemma, the possible forms of $N(abc)$ can be restricted. For convenience, in the lemmas below let $abc \in X_1$ and $N:=N(abc)$.

\begin{lem} \label{lma3}
    \begin{enumerate}
        \item If the row sums of $N$ are $\{ 2,5,5 \}$, then $N$ is 
        \begin{align}
            \begin{matrix}
                & 1 & 1 & 1 & 1 & 1 & 0 & 0\\
                & 1 & 1 & 1 & 1 & 1 & 0 & 0\\
                & 1 & 1 & 0 & 0 & 0 & 0 & 0\\
            \end{matrix} \label{nbarray1}\\
            \nonumber \\ 
            \begin{matrix}
                & 1 & 1 & 1 & 1 & 1 & 0 & 0\\
                \text{or }& 1 & 1 & 1 & 1 & 0 & 1 & 0\\
                & 1 & 1 & 0 & 0 & 0 & 0 & 0\\
            \end{matrix} \label{nbarray2}
        \end{align}
        
        \item If the row sums of $N$ are $\{ 3,3,6 \}$, then $N$ is
        \begin{align}
            \begin{matrix}
                & 1 & 1 & 1 & 1 & 1 & 1 & 0\\
                & 1 & 1 & 1 & 0 & 0 & 0 & 0\\
                & 1 & 1 & 0 & 1 & 0 & 0 & 0\\
            \end{matrix} \label{nbarray3}\\
            \nonumber \\ 
            \begin{matrix}
                & 1 & 1 & 1 & 1 & 1 & 1 & 0\\
                \text{or }& 1 & 1 & 1 & 0 & 0 & 0 & 0\\
                & 1 & 1 & 1 & 0 & 0 & 0 & 0\\
            \end{matrix} \label{nbarray4}
        \end{align}
        
        \item If the row sums of $N$ are $\{ 0,6,6 \}$, then $N$ is
        \begin{align}
            \begin{matrix}
                & 1 & 1 & 1 & 1 & 1 & 1 & 0\\
               & 1 & 1 & 1 & 1 & 1 & 1 & 0\\
                & 0 & 0 & 0 & 0 & 0 & 0 & 0\\
            \end{matrix} \label{nbarray5}\\
            \nonumber \\ 
            \begin{matrix}
                & 1 & 1 & 1 & 1 & 1 & 1 & 0\\
                \text{or }& 1 & 1 & 1 & 1 & 1 & 0 & 1\\
                & 0 & 0 & 0 & 0 & 0 & 0 & 0\\
            \end{matrix} \label{nbarray6}
        \end{align}

        \item If the row sums of $N$ are $\{ 4,4,4 \}$, then $N$ is
        \begin{align}
            \begin{matrix}
                & 1 & 1 & 1 & 1 & 0 & 0 & 0\\
               & 1 & 1 & 1 & 1 & 0 & 0 & 0\\
                & 1 & 1 & 1 & 1 & 0 & 0 & 0\\
            \end{matrix} \label{nbarray7}\\
            \nonumber \\ 
            \begin{matrix}
                & 1 & 1 & 1 & 1 & 0 & 0 & 0\\
                \text{or }& 1 & 1 & 1 & 0 & 1 & 0 & 0\\
                & 1 & 1 & 1 & 0 & 1 & 0 & 0\\
            \end{matrix} \label{nbarray8}\\
            \nonumber \\ 
            \begin{matrix}
                & 1 & 1 & 1 & 1 & 0 & 0 & 0\\
                \text{or }& 1 & 1 & 1 & 0 & 1 & 0 & 0\\
                & 1 & 1 & 1 & 0 & 0 & 1 & 0\\
            \end{matrix} \label{nbarray9}\\
            \nonumber \\ 
            \begin{matrix}
                & 1 & 1 & 1 & 1 & 0 & 0 & 0\\
                \text{or }& 1 & 1 & 1 & 0 & 1 & 0 & 0\\
                & 1 & 1 & 0 & 1 & 1 & 0 & 0\\
            \end{matrix} \label{nbarray10}
        \end{align}

        \item If the row sums of $N$ are $\{ 1,4,7 \}$, then $N$ is
        \begin{align}
            \begin{matrix}
                & 1 & 1 & 1 & 1 & 1 & 1 & 1\\
               & 1 & 1 & 1 & 1 & 0 & 0 & 0\\
                & 1 & 0 & 0 & 0 & 0 & 0 & 0\\
            \end{matrix} \label{nbarray11}
        \end{align}
    \end{enumerate}
\end{lem}

Now let us show that for each of the five situations above, all but one form of $N$ lead to a contradiction when considering other points than $abc$.

 
\begin{lem}
    If $N$ has row sums $\{2,5,5\}$, then $N$ is of the form \eqref{nbarray2}.
\end{lem}
\begin{proof}
Suppose that $N$ has the form \eqref{nbarray1}:
$$
\begin{matrix}
                    & 1 & 1 & 1 & 1 & 1 & 0 & 0 & \gets ab\text{-row}\\
    N : & 1 & 1 & 1 & 1 & 1 & 0 & 0 & \gets ac\text{-row}\\
                    & 1 & 1 & 0 & 0 & 0 & 0 & 0 & \gets bc\text{-row}\\
                    & \uparrow & \uparrow & \uparrow & \uparrow & \uparrow & \uparrow & \uparrow & \\
                    & d & e & f & g & h & i & j
\end{matrix}
$$
From \Cref{lma1} we have $1 = \frac{1}{3} (\ol{ac*} - \ol{bc*}) = \ol{bdi} - \ol{adi}$, so $\ol{bdi} = 1$. Similarly $\ol{bdj} = 1$.\\

Thus
$$
\begin{matrix}
                    & 1 & 1 & 1 & 1 & 1 & 0 & 0 & \gets ab\text{-row}\\
    N(abd) : & * & * & * & * & * & * & * & \gets ad\text{-row}\\
                    & * & * & * & * & * & 1 & 1 & \gets bd\text{-row}\\
                    & \uparrow & \uparrow & \uparrow & \uparrow & \uparrow & \uparrow & \uparrow & \\
                    & c & e & f & g & h & i & j
\end{matrix}
$$
where $*$ is either 0 or 1. But $abd \in X_1$ and this does not fit any nb-array from \Cref{lma3}, so we have a contradiction. (\Cref{lma4} also gives a contradiction.)  
\end{proof}


\begin{lem}
    If $N$ has row sums $\{3,3,6\}$, then $N$ is of the form \eqref{nbarray3}.
\end{lem}
\begin{proof}
Suppose that $N$ has the form \eqref{nbarray4}:
$$
\begin{matrix}
                    & 1 & 1 & 1 & 1 & 1 & 1 & 0 & \gets ab\text{-row}\\
    N : & 1 & 1 & 1 & 0 & 0 & 0 & 0 & \gets ac\text{-row}\\
                    & 1 & 1 & 1 & 0 & 0 & 0 & 0 & \gets bc\text{-row}\\
                    & \uparrow & \uparrow & \uparrow & \uparrow & \uparrow & \uparrow & \uparrow & \\
                    & d & e & f & g & h & i & j
\end{matrix}
$$
Then from \Cref{lma1} we have $1 = \frac{1}{3}(\ol{ab*} - \ol{ac*}) = \ol{cde} - \ol{bde}$ so $\ol{cde} = 1$ and $\ol{bde} = 0$. Also $0 = \frac{1}{3}(\ol{ac*} - \ol{bc*}) = \ol{bde} - \ol{ade}$, so $\ol{ade} = 0$. For a similar reason $\ol{adf} = 0$.\\

Thus
$$
\begin{matrix}
                    & 1 & 1 & 1 & 0 & 0 & 0 & 0 & \gets ac\text{-row}\\
    N(acd) : & 1 & 0 & 0 & * & * & * & * & \gets ad\text{-row}\\
                    & * & * & * & * & * & * & * & \gets cd\text{-row}\\
                    & \uparrow & \uparrow & \uparrow & \uparrow & \uparrow & \uparrow & \uparrow & \\
                    & b & e & f & g & h & i & j
\end{matrix}
$$
But $acd \in X_1$ and this does not fit any nb-array from \Cref{lma3}, so we have a contradiction. (\Cref{lma4} also gives a contradiction.)
\end{proof}


\begin{lem}
    If $N$ has row sums $\{0,6,6\}$, then $N$ is of the form \eqref{nbarray6}.
\end{lem}
\begin{proof}
Suppose that $N$ has the form \eqref{nbarray5}:
$$
\begin{matrix}
                    & 1 & 1 & 1 & 1 & 1 & 1 & 0 & \gets ab\text{-row}\\
    N : & 1 & 1 & 1 & 1 & 1 & 1 & 0 & \gets ac\text{-row}\\
                    & 0 & 0 & 0 & 0 & 0 & 0 & 0 & \gets bc\text{-row}\\
                    & \uparrow & \uparrow & \uparrow & \uparrow & \uparrow & \uparrow & \uparrow & \\
                    & d & e & f & g & h & i & j
\end{matrix}
$$
Then from \Cref{lma1}, $2 = \frac{1}{3} (\ol{ac*} - \ol{bc*}) = 1 + \ol{bdj} - \ol{adj}$, so $\ol{adj} = 0$ and $\ol{bdj}=1$.\\

Thus
$$
\begin{matrix}
                    & 1 & 1 & 1 & 1 & 1 & 1 & 0 & \gets ab\text{-row}\\
    N(abd) : & 1 & * & * & * & * & * & 0 & \gets ad\text{-row}\\
                    & 0 & * & * & * & * & * & 1 & \gets bd\text{-row}\\
                    & \uparrow & \uparrow & \uparrow & \uparrow & \uparrow & \uparrow & \uparrow & \\
                    & c & e & f & g & h & i & j
\end{matrix}
$$
Which does not fit any nb-array from \Cref{lma3} so we have a contradiction.
\end{proof}


\begin{lem}
    If $N$ has row sums $\{4,4,4\}$, then $N$ is of the form \eqref{nbarray10}.
\end{lem}
\begin{proof}
First we eliminate \eqref{nbarray8}. Suppose that $N$ has the form \eqref{nbarray8}:
$$
\begin{matrix}
                    & 1 & 1 & 1 & 1 & 0 & 0 & 0 & \gets ab\text{-row}\\
    N : & 1 & 1 & 1 & 0 & 1 & 0 & 0 & \gets ac\text{-row}\\
                    & 1 & 1 & 1 & 0 & 1 & 0 & 0 & \gets bc\text{-row}\\
                    & \uparrow & \uparrow & \uparrow & \uparrow & \uparrow & \uparrow & \uparrow & \\
                    & d & e & f & g & h & i & j
\end{matrix}
$$
Then from \Cref{lma1} we have $0 = \frac{1}{3} (\ol{ab*} - \ol{ac*}) = 1 + \ol{cdg} - \ol{bdg}$ so $\ol{bdg} = 1$ and $\ol{cdg} = 0$. Also $0 = \frac{1}{3} (\ol{ac*} - \ol{bc*}) = \ol{bdg} - \ol{adg}$, so $\ol{adg} = 1$. Furthermore $0 = \frac{1}{3} (\ol{ab*} - \ol{ac*}) = -1 + \ol{cdh} - \ol{bdh}$ so $\ol{cdh} = 1$ and $\ol{bdh}= 0$. And $0 = \frac{1}{3} (\ol{ac*} - \ol{bc*}) = \ol{bdh} - \ol{adh}$ so $\ol{adh} = 0$. Lastly $0 = \frac{1}{3} (\ol{ab*} - \ol{bc*}) = \ol{cde} - \ol{ade} = \ol{cdf} - \ol{adf}$ so $\ol{cde} = \ol{ade}$ and $\ol{cdf} = \ol{adf}$\\

Thus
$$
\begin{matrix}
                    & 1 & 1 & 1 & 0 & 1 & 0 & 0 & \gets ac\text{-row}\\
    N(acd) : & * & \alpha_1 & \alpha_2 & 1 & 1 & * & * & \gets ad\text{-row}\\
                    & * & \alpha_1 & \alpha_2 & 0 & 0 & * & * & \gets cd\text{-row}\\
                    & \uparrow & \uparrow & \uparrow & \uparrow & \uparrow & \uparrow & \uparrow & \\
                    & b & e & f & g & h & i & j
\end{matrix}
$$
for some $\alpha_1, \alpha_2 \in \{0,1\}$. And since $acd \in X_1$ and this does not match any pattern from \Cref{lma3}, we have a contradiction.\\

Now we eliminate the case \eqref{nbarray7}.
Suppose that $N$ has the form \eqref{nbarray7}:
$$
\begin{matrix}
                    & 1 & 1 & 1 & 1 & 0 & 0 & 0 & \gets ab\text{-row}\\
    N : & 1 & 1 & 1 & 1 & 0 & 0 & 0 & \gets ac\text{-row}\\
                    & 1 & 1 & 1 & 1 & 0 & 0 & 0 & \gets bc\text{-row}\\
                    & \uparrow & \uparrow & \uparrow & \uparrow & \uparrow & \uparrow & \uparrow & \\
                    & d & e & f & g & h & i & j
\end{matrix}
$$
Then from \Cref{lma1} it follows that for distinct $x,y \in \{d, e, f, g, h, i, j\}$, $\ol{axy} = \ol{bxy} = \ol{cxy}$.\\

Thus
$$
\begin{matrix}
                    & 1 & 1 & 1 & 1 & 1 & 0 & 0 & \gets ab\text{-row}\\
    N(abh) : & 0 & \alpha_1 & \alpha_2 & \alpha_3 & \alpha_4 & \alpha_5 & \alpha_6 & \gets ah\text{-row}\\
                    & 0 & \alpha_1 & \alpha_2 & \alpha_3 & \alpha_4 & \alpha_5 & \alpha_6 & \gets bh\text{-row}\\
                    & \uparrow & \uparrow & \uparrow & \uparrow & \uparrow & \uparrow & \uparrow & \\
                    & c & d & e & f & g & i & j
\end{matrix}
$$
for some $\alpha_1, \dots, \alpha_6 \in \{0,1\}$. Since $abh \in X_2$ the sum of $N(abh)$ must be 9, and from \Cref{lma1} the difference between two row sums must be a multiple of 3. Thus the row sums of $N(abh)$ are $\{2,2,5\}$, and from \Cref{lma4} we must have $\alpha_5 = 0$ and $\alpha_6 = 0$. Also, two of $\alpha_1, \dots, \alpha_4$ must be 1. Without loss of generality say $\alpha_1 = 1$, so $\ol{adh} = \ol{bdh} = 1$.\\

Then 
$$
\begin{matrix}
                    & 1 & 1 & 1 & 1 & 0 & 0 & 0 & \gets ab\text{-row}\\
    N(abd) : & 1 & \beta_1 & \beta_2 & \beta_3 & 1 & \beta_4 & \beta_5 & \gets ah\text{-row}\\
                    & 1 & \beta_1 & \beta_2 & \beta_3 & 1 & \beta_4 & \beta_5 & \gets bh\text{-row}\\
                    & \uparrow & \uparrow & \uparrow & \uparrow & \uparrow & \uparrow & \uparrow & \\
                    & c & e & f & g & h & i & j
\end{matrix}
$$
where $\beta_1, \dots, \beta_5 \in \{0,1\}$. But since $abd \in X_1$, $N(abd)$ must have the form \eqref{nbarray8} that we eliminated before. So we have a contradiction.\\

Finally, we eliminate the case \eqref{nbarray9}. Suppose that $N$ has the form \eqref{nbarray9}:
$$
\begin{matrix}
                    & 1 & 1 & 1 & 1 & 0 & 0 & 0 & \gets ab\text{-row}\\
    N : & 1 & 1 & 1 & 0 & 1 & 0 & 0 & \gets ac\text{-row}\\
                    & 1 & 1 & 1 & 0 & 0 & 1 & 0 & \gets bc\text{-row}\\
                    & \uparrow & \uparrow & \uparrow & \uparrow & \uparrow & \uparrow & \uparrow & \\
                    & d & e & f & g & h & i & j
\end{matrix}
$$
Then from \Cref{lma1} we have $0=\frac{1}{3} (\ol{ab*}-\ol{bc*})=1+\ol{cgh}-\ol{agh}=1+\ol{cgj}-\ol{agj}$, so $\ol{agh} = 1$ and $\ol{agj} = 1$. Also $0=\frac{1}{3} (\ol{ab*}-\ol{ac*})= 1 + \ol{cdg} - \ol{bdg}$, so $\ol{bdg} = 1$.\\

Thus 
$$
\begin{matrix}
                    & 1 & 1 & 1 & 1 & 0 & 0 & 0 & \gets ab\text{-row}\\
    N(abg) : & 1 & * & * & * & 1 & * & 1 & \gets ag\text{-row}\\
                    & 1 & 1 & * & * & * & * & * & \gets bg\text{-row}\\
                    & \uparrow & \uparrow & \uparrow & \uparrow & \uparrow & \uparrow & \uparrow & \\
                    & c & d & e & f & h & i & j
\end{matrix}
$$
which, since $abg \in X_1$, gives a contradiction.
\end{proof}

Now only five cases remain from \Cref{lma3}.

\begin{proposition} \label{prop1}
    We have the following.
    \begin{enumerate}
        \item If $X_1$ is of type 0, then for any $abc \in X_1$, $N(abc)$ has the form \eqref{nbarray2}.
        \item If $X_1$ is of type 1, then for any $abc \in X_1$, $N(abc)$ has the form \eqref{nbarray3} or \eqref{nbarray6}.
        \item If $X_1$ is of type 2, then for any $abc \in X_1$, $N(abc)$ has the form \eqref{nbarray10} or \eqref{nbarray11}.\\
    \end{enumerate}
\end{proposition}

These cases occur in the construction depicted at the beginning of the section.\\

For the first perfect $2$-coloring $\{P_1, P_2\}$, $P_1$ is of type 0 and $P_2$ is of type 2. When taking $X_1 = P_1$, all of its vertices have an nb-array of the form \eqref{nbarray2}. When taking $X_1 = P_2$, vertices from $E$ and $G$ have an nb-array of the form \eqref{nbarray10}, while vertices from $D$ and $F$ have an nb-array of the form \eqref{nbarray11}.\\

For the second perfect $2$-coloring $\{P_1', P_2'\}$, $P_1'$ and $P_2'$ are both of type 1. When taking $X_1 = P_1'$, vertices from $C$, $D$ and $H$ have an nb-array of the form \eqref{nbarray3} and those from $G$ have the form \eqref{nbarray6}. When taking $X_1 = P_2'$, vertices from $A$, $B$ and $F$ have an nb-array of the form \eqref{nbarray3} and those from $E$ have the form \eqref{nbarray6}.\\

\begin{theorem} \label{thm1}
    There is only one perfect $2$-coloring with a part of type 0, which is also the only perfect $2$-coloring with a part of type 2, up to isomorphism.
\end{theorem}

\begin{proof}
Denote $t_1$ the type of $X_1$ and $t_2$ the type of $X_2$. Since $|ab* \cap X_1| + |ab* \cap X_2| = |ab*| = 8$, we have $t_1 + t_2 \equiv 2 \pmod 3$. So if $X_1$ is of type 0 then $X_2$ is of type 2, and vice versa. Therefore, it is enough to show that the perfect $2$-coloring with a part of type 0 is unique.\\

Suppose that $X_1$ is of type 0, and fix $abc \in X_1$. Then from \Cref{prop1}
$$
\begin{matrix}
                    & 1 & 1 & 1 & 1 & 1 & 0 & 0 & \gets ab\text{-row}\\
    N(abc) : & 1 & 1 & 1 & 1 & 0 & 1 & 0 & \gets ac\text{-row}\\
                    & 1 & 1 & 0 & 0 & 0 & 0 & 0 & \gets bc\text{-row}\\
                    & \uparrow & \uparrow & \uparrow & \uparrow & \uparrow & \uparrow & \uparrow & \\
                    & d & e & f & g & h & i & j
\end{matrix}
$$

Since the rest of the proof relies on many applications of \Cref{lma1} and is quite fastidious, we will leave the verification to a computer.\\

We consider the 120 values $\ol{xyz}$ ($ xyz \in J(10,3) $) as variables for multivariate polynomials. The variables having $\{0,1\}$ values translates to $\ol{xyz}^2-\ol{xyz}=0$. \Cref{lma1} and the values fixed in $N(abc)$ also give some multivariate polynomials that must have value 0. The ideal generated by these polynomials can be computed by magma.\\

We can then check if $\ol{xyz} = \epsilon$ ($\epsilon \in \{0,1\}$) can be deduced by checking if $\ol{xyz} - \epsilon$ belongs to the ideal. Moreover, we can check if $\ol{xyz} = \ol{x'y'z'}$ or $\ol{xyz} \neq \ol{x'y'z'}$ ($ xyz $, $x'y'z' \in J(10,3) $) can be deduced by checking if $\ol{xyz} - \ol{x'y'z'}$ or $\ol{xyz} + \ol{x'y'z'} -1$ belongs to the ideal.\\

In this way, 48 of the $\ol{xyz}$ values are deduced to be 1, and 48 of them are 0. The remaining ones are separated into two groups $U_1$ and $U_2$ of size 12, with identical value within a group. Since $|X_1|=|X_2|=60$, there are two possibilities for the coloring $(X_1,X_2)$, either the values in the group $U_1$ are 0 and those in the group $U_2$ are 1, or the opposite. But it is computed in the magma code that the transposition $(f,g)$ is an isomorphism between those two possible colorings.\\

\begin{verbatim}
J103:=Setseq(Subsets({1..10},3));
R:=PolynomialRing(Rationals(),#J103);
zo:={ R.i^2-R.i : i in {1..#J103} };
pos:=func< a,b,c | R.Position(J103,{a,b,c}) >;
abstar:=func< a,b | &+[ pos(a,b,x) : x in {1..10} | not x in {a,b} ] >;
lem1:=func< a,b,c,d,e | abstar(a,b)-abstar(a,c)-3*(
 pos(a,b,d)+pos(a,b,e)+pos(c,d,e)-pos(a,c,d)-pos(a,c,e)-pos(b,d,e) ) >;
lem1s:={ lem1(a,b,c,d,e) : a,b,c,d,e in {1..10} | #{a,b,c,d,e} eq 5 };
prop12:={ pos(xyz[1],xyz[2],xyz[3])-1 : xyz in 
  { [1,2,3],[1,2,4],[1,2,5],[1,2,6],[1,2,7],[1,2,8],
   [1,3,4],[1,3,5],[1,3,6],[1,3,7],[1,3,9],
   [2,3,4],[2,3,5] } } join 
 { pos(xyz[1],xyz[2],xyz[3]) : xyz in 
  { [1,2,9],[1,2,10],[1,3,8],[1,3,10],[2,3,6],[2,3,7],[2,3,8],[2,3,9],
   [2,3,10] } };
I:=ideal< R | zo join lem1s join prop12 >;

known1:={ {a,b,c} : a,b,c in {1..10} | #{a,b,c} eq 3 and pos(a,b,c)-1 in I };
#known1 eq 48;
known0:={ {a,b,c} : a,b,c in {1..10} | #{a,b,c} eq 3 and pos(a,b,c) in I };
#known0 eq 48;

pos(1,4,6)+pos(1,4,7)-1 in I;

unknown1:={ {a,b,c} : a,b,c in {1..10} | #{a,b,c} eq 3 and
 pos(a,b,c)-pos(1,4,6) in I};
unknown2:={ {a,b,c} : a,b,c in {1..10} | #{a,b,c} eq 3 and
 pos(a,b,c)-pos(1,4,7) in I};
#unknown1 eq 12;
#unknown2 eq 12;
#(known0 join known1 join unknown1 join unknown2) eq 120;

t:=Sym(10)!(6,7);
G:=GSet(Sym(10),Subsets({1..10},3));
{Image(t,G,X) : X in known0} eq known0;
{Image(t,G,X) : X in known1} eq known1;
{Image(t,G,X) : X in unknown1} eq unknown2;
\end{verbatim}

\end{proof}

\begin{theorem}  \label{thm2}
    There is only one perfect $2$-coloring with a part of type 1 up to isomorphism.
\end{theorem}

\begin{proof}
Suppose that $X_1$ is of type 1 and fix $abc \in X_1$. Then from \Cref{prop1}, $N(abc)$ has the form \eqref{nbarray3} or \eqref{nbarray6}. If $N(abc)$ has the form \eqref{nbarray6}, then the second part of the magma code shows that there exists an element of $X_1$ which has an nb-array of the form \eqref{nbarray3}. The case \eqref{nbarray3} is very similar to the proof of \Cref{prop1} so we will give an almost identical magma code.\\

\begin{verbatim}
// Case (3)

prop14:={ pos(xyz[1],xyz[2],xyz[3])-1 : xyz in 
  { [1,2,3],[1,2,4],[1,2,5],[1,2,6],[1,2,7],[1,2,8],[1,2,9],
   [1,3,4],[1,3,5],[1,3,6],
   [2,3,4],[2,3,5],[2,3,7] } } join 
 { pos(xyz[1],xyz[2],xyz[3]) : xyz in 
  { [1,2,10],
   [1,3,7],[1,3,8],[1,3,9],[1,3,10],
   [2,3,6],[2,3,8],[2,3,9],[2,3,10] } };
I:=ideal< R | zo join lem1s join prop14 >;

known1:={ {a,b,c} : a,b,c in {1..10} | #{a,b,c} eq 3 and pos(a,b,c)-1 in I };
#known1 eq 48;
known0:={ {a,b,c} : a,b,c in {1..10} | #{a,b,c} eq 3 and pos(a,b,c) in I };
#known0 eq 48;

pos(1,5,8)+pos(1,5,9)-1 in I;

unknown1:={ {a,b,c} : a,b,c in {1..10} | #{a,b,c} eq 3 and
 pos(a,b,c)-pos(1,5,8) in I};
unknown2:={ {a,b,c} : a,b,c in {1..10} | #{a,b,c} eq 3 and
 pos(a,b,c)-pos(1,5,9) in I};
#unknown1 eq 12;
#unknown2 eq 12;
#(known0 join known1 join unknown1 join unknown2) eq 120;

t:=Sym(10)!(8,9);
G:=GSet(Sym(10),Subsets({1..10},3));
{Image(t,G,X) : X in known0} eq known0;
{Image(t,G,X) : X in known1} eq known1;
{Image(t,G,X) : X in unknown1} eq unknown2;

// Case (6)

prop14prime:={ pos(xyz[1],xyz[2],xyz[3])-1 : xyz in 
  { [1,2,3],[1,2,4],[1,2,5],[1,2,6],[1,2,7],[1,2,8],[1,2,9],
   [1,3,4],[1,3,5],[1,3,6],[1,3,7],[1,3,8],[1,3,10] } } join 
 { pos(xyz[1],xyz[2],xyz[3]) : xyz in 
  { [1,2,10],
   [1,3,9],
   [2,3,4],[2,3,5],[2,3,6],[2,3,7],[2,3,8],[2,3,9],[2,3,10] } };
Iprime:=ideal< R | zo join lem1s join prop14prime >;

abstar(3,4)-4 in Iprime;
abstar(4,10)-4 in Iprime;
abstar(3,10)-7 in Iprime;

\end{verbatim}
\end{proof}

\section{Conclusion}

By a proof similar to \Cref{thm1}, we can see that a perfect $2$-coloring with a part of type 1 must have the other part also of type 1. Thus, these last two theorems show that there are only two symmetric perfect $2$-colorings of $J(10,3)$ associated to $\theta_2$. One of them has parts of type 0 and 2, and the other one has both parts of type 1. \Cref{thm2} also implies that the two parts of the perfect $2$-coloring with parts of type 1 are isomorphic to each other.

\section{Acknowledgement}

I am grateful to Prof. Akihiro Munemasa for his continuous guidance through this study, and his idea to consider an ideal membership problem for the resolutions by computer.\\

I am also thankful to Prof. Alexander Gavrilyuk and Prof. Sergey Goryainov for their insight on the known constructions of symmetric perfect $2$-colorings of $J(10,3)$, and for indicating me that the problem resolved in this paper was still open.\\

This work was supported by JST SPRING (Grant Number JPMJSP2114).

\end{document}